\documentclass{monsky2009}

\usepackage{amssymb}

\usepackage{bm}

\newenvironment{conjecture*}[1][]{\textbf{Conjecture #1\hspace{.3em}}}{}
\newenvironment{theorem*}[1]{\textbf{#1}\itshape \hspace{.3em}}{\upshape}
\newenvironment{remark*}[1]{\textbf{#1}\itshape \hspace{.3em}}{\upshape}
\newenvironment{example*}[1]{\textbf{#1}\itshape \hspace{.3em}}{\upshape}
\newenvironment{proof}[1][]{\textbf{Proof #1\hspace{.3em}}}{}

\newtheorem{definition}{Definition}[section]
\newtheorem{theorem}[definition]{Theorem}
\newtheorem{lemma}[definition]{Lemma}

\newtheorem{corollary}[definition]{Corollary}
\newtheorem{remark}[definition]{Remark}

\newtheorem{observation}[definition]{Observation}

\newcounter{kpremark}
\newcommand{\mod}[1]{\ensuremath{\hspace{.5em}(#1)}}

\providecommand{\comp}{\ensuremath\raisebox{.2em}{\,\,\mbox{$\scriptstyle\circ$}}\,\,}

\begin{document}

\addtolength{\parskip}{-2pt}

\begin{frontmatter}






\title{Variations on a Lemma of Nicolas and Serre}
\author{Paul Monsky}

\address{Brandeis University, Waltham MA  02454-9110, USA. monsky@brandeis.edu}

\begin{abstract}
The ``Nicolas-Serre code'',  $(a,b) \leftrightarrow t^{n}$, is a bijection between $N\times N$ and those $t^{n}$, $n$ odd, in $Z/2[t]$. Suppose $A_{n}$, $n$ odd, in $Z/2[t]$ are defined by: $A_{1}= A_{5}= 0$, $A_{3}= t$, $A_{7}= t^{5}$, and $A_{n+8}= t^{8} A_{n} + t^{2} A_{n+2}$. A lemma, Proposition 4.3 of \cite{6}, used to study the Hecke algebra attached to the space of mod $2$ level $1$ modular forms, gives information about the codes $(a,b)$ attached to the monomials appearing in $A_{n}$. The unpublished highly technical proof has been simplified by Gerbelli-Gauthier.

Our Theorem \ref{theorem3.7} generalizes Proposition 4.3. The proof, in sections \ref{section1}--\ref{section3}, is a further simplification of Gerbelli-Gauthier's argument. We build up to the theorem with variants involving the same recurrence, but having different sorts of initial conditions.

Section \ref{section4} treats the recurrence $A_{n+16}= t^{16} A_{n} + t^{4} A_{n+4} + t^{2} A_{n+2}$. Theorem \ref{theorem4.1}, the analog to Theorem \ref{theorem3.7} for this recurrence, is used in \cite{2} and \cite{3} to analyze level 3 Hecke algebras.

Finally we introduce a variant code, $(a,b) \leftrightarrow w^{n}$ which is a bijection between $N\times N$ and those $w^{n}$, $n \equiv 1,3,7,9 \bmod{20}$,  in $Z/2[w]$. We then study the recurrence $A_{n+80}= w^{80} A_{n}+ w^{20} A_{n+20}$, $n \equiv 1,3,7,9 \bmod{20}$, with appropriate initial conditions. Lemma \ref{lemma5.5}, derived from the results of sections \ref{section1}--\ref{section3}, is the precise analog of Proposition 4.3 for this code, this recurrence, and these initial conditions. It is used in \cite{4} and \cite{5} to analyze level 5 Hecke algebras.

\end{abstract}


\end{frontmatter}


\section{Introduction}
\label{section1}

Let $r$ be a power of $2$. Suppose that for each odd $n>0$ we have an $A_{n}$ in $tZ/2[t^{2}]$, and that these satisfy 
\[
(\star)\qquad
A_{n+8r} = t^{8r}A_{n}+t^{2r}A_{n+2r}.
\]

Recursions of this sort arise in the study of the action of the Hecke operator $T_{3}$ on certain spaces of mod~$2$ modular forms. Results about the exponents appearing in $A_{n}$ (and in the $A_{n}$ satisfying similar recursions) have been used by Nicolas and Serre \cite{6} to understand the Hecke algebra attached to the space of mod $2$ modular forms of level $1$. These questions are often difficult (Nicolas and Serre have a proof about a recursion attached to $T_{5}$ that seems quite inscrutable). The questions I have in mind involve the ``Nicolas-Serre code'', and the proofs rely on the easily proved but central:

\begin{observation}
\label{observation1.1}
If the $A_{n}$ satisfy $(\star)$ for $r$, then they also satisfy $(\star)$ when $r$ is replaced by any larger power of $2$.
\end{observation}

Proposition 4.3 of \cite{6} is the special case of $(\star)$ with $r=1$, $A_{1}=0$, $A_{3}=t$, $A_{5}=0$, $A_{7}=t^{5}$. Nicolas and Serre's highly technical treatment of it was greatly simplified by Gerbelli-Gauthier \cite{1}.  In this note I use Gerbelli-Gauthier's technique (with some further simplifications) to treat various other initial values for the $A_{n}$. In each case I get a result concerning the ``dominant term'' appearing in the Nicolas-Serre code for $A_{n}$.

Theorem \ref{theorem3.1}, stated at the end of this section and proved in section \ref{section3} is a bit simpler than Proposition 4.3, and its proof nicely illustrates the general ideas. Three variants, Theorems \ref{theorem3.3}, \ref{theorem3.5} and \ref{theorem3.7}, follow. The last of these has (the algebraic form of) Proposition 4.3 as a corollary.

When one tries to treat modular forms of level $\Gamma_{0}(3)$ similarly, one must deal with the fact that the recursion attached to the Hecke operator $U_{3}$ is nothing like $(\star)$. (But this recursion also leads to results of interest---see \cite{3}.) Instead one may use a recursion attached to $T_{7}$:

\[
(\star\star)\qquad
A_{n+16r} = t^{16r}A_{n}+t^{4r}A_{n+4r}+t^{2r}A_{n+2r}.
\]

Under suitable initial conditions there is an analog to Proposition 4.3 for $(\star\star)$. I'll prove such a result, Corollary \ref{corollary4.2}, in section \ref{section4}. Applications to Hecke algebras appear in \cite{2} and \cite{3}. The arguments are motivated once again by Gerbelli-Gauthier's technique; I'm grateful to her for her ideas. In the final sections I define a variant of the Nicolas-Serre code, and prove one more Proposition 4.3 analog, using the corollaries to Theorems \ref{theorem3.1}, \ref{theorem3.3}, \ref{theorem3.5} and \ref{theorem3.7}. This result, Lemma \ref{lemma5.5}, has the same role in treating level $\Gamma_{0}(5)$ as Corollary \ref{corollary4.2} does in treating level $\Gamma_{0}(3)$; see \cite{4} and \cite{5}.

We now introduce the Nicolas-Serre code in language that differs slightly from that of \cite{6}.

\begin{definition}
\label{def1.2}
$g:N\rightarrow N$ is the function with $g(2n)=4g(n)$, $g(2n+1)=g(2n)+1$.
\end{definition}

Since $g(0)=4g(0)$, $g(0)=0$, and the functional equations above give all the $g(n)$. One sees immediately that if $n$ is a sum of distinct $q$, with each $q$ a power of $2$, then $g(n)$ is the sum of the $q^{2}$. It follows that each $n$ in $N$ can be uniquely written as $g(a)+2g(b)$ for some $a$ and $b$ in $N$. So if we let $[a,b]$ be the monomial $t^{n}$, where $n=1+2g(a)+4g(b)$, then $[a,b]\rightarrow [a,b]$ maps $N\times N$ $1-1$ onto the monomials $t^{k}$, $k$ odd. It follows that:

\[
[0,0],[1,0],[0,1],[2,0],[1,1],[0,2],[3,0],[2,1],[1,2],[0,3],\ldots
\]

is a list of all such monomials.

\begin{definition}
\label{def1.3}
$t^{n}$ ``precedes'' or ``is earlier than'' $t^{m}$ if it appears before it in the above list. In other words to say that $[c,d]$ precedes $[a,b]$ is to say that $c+d\le a+b$, and that, in case of equality, $d<b$.
\end{definition}

\begin{theorem*}{Theorem 3.1}
\label{theorem3.1prelim}
Suppose $A_{n}$ in $tZ/2[t^{2}]$ satisfy $(\star)$ of our first paragraph for some $r$. Suppose further that the following condition  $(1\gamma)$ holds when $n<8r$:

\[
(1\gamma)\qquad
A_{n}\mbox{ is a sum of monomials preceding }t^{n}.
\]

Then $(1\gamma)$ holds for all $n$.
\end{theorem*}

\begin{remark}
\label{remark1.4}
Each of $8r, 16r, 32r, \ldots$ can be written as $8q^{2}$ or $16q^{2}$ for some $q$ which is a power of $2$. So Theorem \ref{theorem3.1} is a consequence of the following:

Suppose $q$ is a power of $2$ and that $A_{n}$, $n>0$ and odd, are in $tZ/2[t^{2}]$. Then:

\begin{enumerate}
\item[(1)] If the $A_{n}$ satisfy $(\star)$ for some $r$ dividing $q^{2}$ (and therefore for $r=q^{2}$ by Observation \ref{observation1.1}), and $(1\gamma)$ holds when $n<8q^{2}$, then it holds when $n<16q^{2}$.
\item[(2)] If the $A_{n}$ satisfy $(\star)$ for some $r$ dividing $2q^{2}$, and $(1\gamma)$ holds when $n<16q^{2}$, then it holds when $n<32q^{2}$.
\end{enumerate}
\end{remark}

I'll say now a few words about the proof of (1). We argue by induction on $n$. So suppose $k$ is in $(8q^{2},16q^{2})$---we want to show that $A_{k}$ is a sum of monomials preceding $t^{k}$. Let $n=k-8q^{2}$ and write $t^{n}$ as $[a,b]$. Then $1+2g(a)+4g(b)=n<8q^{2}$, and it follows that $a<2q$. So $g(a+2q) = g(a) +4q^{2}$, and replacing $a$ by $a+2q$ in $1+2g(a)+4g(b)$ increases it by $8q^{2}$. So $t^{k} =t^{8q^{2}}[a,b]=[a+2q,b]$, and we want to show that $A_{k}$ is a sum of monomials preceding $[a+2q,b]$. Since $(\star)$ holds for $r=q^{2}$, $A_{k}=t^{8q^{2}}A_{n}+t^{2q^{2}}A_{n+2q^{2}}$, and it's enough to show that $t^{8q^{2}}A_{n}$ and $t^{2q^{2}}A_{n+2q^{2}}$ are such monomial sums. The simple machinery used to prove that this is true is developed in the next section.

Next I turn to (2). Again we argue by induction on $n$. So suppose $k$ is in $(16q^{2},32q^{2})$---we want to show that $A_{k}$ is a sum of monomials preceding $t^{k}$.  Let $n=k-16q^{2}$, and write $t^{n}$ as $[a,b]$. Then $1+2g(a)+4g(b)=n<16q^{2}$, and $b<2q$. Then $g(b+2q) = g(b) +4q^{2}$, and replacing $b$ by $b+2q$ in $1+2g(a)+4g(b)$ results in an increase of $16q^{2}$. So $t^{k} =t^{16q^{2}}[a,b]=[a, b+2q]$, and we want to show that $A_{k}$ is a sum of monomials preceding $[a, b+2q]$. Since $(\star)$ holds for $r=2q^{2}$, $A_{k}=t^{16q^{2}}A_{n}+t^{4q^{2}}A_{n+4q^{2}}$, and it's enough to show that $t^{16q^{2}}A_{n}$ and $t^{4q^{2}}A_{n+4q^{2}}$ are such monomial sums. Again, the next section furnishes the machinery that's needed.

\section{The effect of multiplication by $\bm t^{2q^{2}}$ and $\bm t^{4q^{2}}$}
\label{section2}

From now on, $q$ is a power of $2$. We start with an easy result.

\begin{theorem}
\label{theorem2.1}
Suppose $\left\lfloor\frac{a}{q}\right\rfloor$ is even. Then $t^{2q^{2}}[a,b]=[a+q,b]$. Suppose $\left\lfloor\frac{b}{q}\right\rfloor$ is even. Then $t^{4q^{2}}[a,b]=[a,b+q]$. 
\end{theorem}

\begin{proof}
If $\left\lfloor\frac{a}{q}\right\rfloor$ is even, $q$ does not occur when $a$ is written as a sum of distinct powers of $2$. So $g(a+q)=g(a)+q^{2}$, and replacing $a$ by $a+q$ in $1+2g(a)+4g(b)$ results in an increase of $2q^{2}$. The second result is proved similarly.\qed
\end{proof}

We assume temporarily that $n=1+2g(a)+4g(b)$ is $1 \bmod{2q^{2}}$. This evidently holds if and only if $q$ divides $a$ and $b$. We shall study $t^{2q^{2}}[a,b]$ and $t^{4q^{2}}[a,b]$ under these assumptions, showing the following: If $\frac{a}{q}$ is odd then $t^{2q^{2}}[a,b]$ is earlier than $[a+q,b]$. If $\frac{b}{q}$ is odd then $t^{4q^{2}}[a,b]$ is earlier than $[a,b+q]$.  

\begin{definition}
\label{def2.2}
Suppose $i$ and $j$ are each in $\{0,q\}$. Then $S_{i,j}$ takes $[a,b]$ to $[2a+i, 2b+j]$.
\end{definition}

\begin{lemma}
\label{lemma2.3}
If $[a,b]=t^{n}$, $S_{0,0}$, $S_{q,0}$, $S_{0,q}$ and $S_{q,q}$ take $[a,b]$ to $t^{4n-3}$, $t^{(4n-3)+2q^{2}}$, $t^{(4n-3)+4q^{2}}$ and $t^{(4n-3)+6q^{2}}$.
\end{lemma}

\begin{proof}
This follows from the fact that $g(2a)=4g(a)$, $g(2a+q)=4g(a)+q^{2}$, $g(2b)=4g(b)$, $g(2b+q)=4g(b)+q^{2}$.\qed
\end{proof}

Note that if $[a,b]$ precedes $[c,d]$ then $S_{i,j}[a,b]$ precedes $S_{i,j}[c,d]$; this is key to what follows.

\begin{lemma}
\label{lemma2.4} \hspace{1em}\\
\vspace{-3ex}
\begin{enumerate}
\item[(1)] $t^{4q^{2}}\comp S_{0,q} = S_{0,0} \comp t^{2q^{2}} = t^{2q^{2}}\comp S_{q,q}$.
\item[(2)] $t^{4q^{2}}\comp S_{q,q} = S_{q,0} \comp t^{2q^{2}}$.
\end{enumerate}

(Here $t^{4q^{2}}$ and $t^{2q^{2}}$ denote the operations of multiplication by $t^{4q^{2}}$ and $t^{2q^{2}}$.)
\end{lemma}

\begin{proof}
Applying the operators of Lemma \ref{lemma2.4} to $t^{n}$, we see from Lemma \ref{lemma2.3} that all the operators in (1) give $t^{(4n-3)+8q^{2}}$, while those in (2) give\linebreak $t^{(4n-3)+10q^{2}}$.\qed
\end{proof}

\begin{lemma}
\label{lemma2.5}
If $\frac{a}{q}$ is odd, $t^{2q^{2}}[a,b]$ precedes $[a+q,b]$.
\end{lemma}

\begin{proof}
We argue by induction on $a+b$. If $\frac{b}{q}$ is even, $g(a-q)=g(a)-q^{2}$,  $g(b+q)=g(b)+q^{2}$, and replacing $a$ and $b$ by $a-q$ and $b+q$ increases $1+2g(a)+4g(b)$ by $2q^{2}$. So $t^{2q^{2}}[a,b]=[a-q,b+q]$ which precedes $[a+q,b]$. Suppose finally that $\frac{a}{q}$ and $\frac{b}{q}$ are odd. Then $[a,b]= S_{q,q}[c,d]$ for some $c$ and $d$, divisible by $q$, and $c+d<a+b$. By Lemma \ref{lemma2.4}, $t^{2q^{2}}[a,b]=S_{0,0}t^{2q^{2}}[c,d]$. Theorem \ref{theorem2.1} and the induction show that $t^{2q^{2}}[c,d]$ equals or precedes $[c+q,d]$. So $t^{2q^{2}}[a,b]$ equals or precedes $[2c+2q,2d]=[a+q, b-q]$, which precedes $[a+q,b]$.\qed
\end{proof}

\begin{lemma}
\label{lemma2.6}
If $\frac{b}{q}$ is odd, $t^{4q^{2}}[a,b]$ precedes $[a,b+q]$.
\end{lemma}

\begin{proof}
If $\frac{a}{q}$ is even, then $[a,b]= S_{0,q}[c,d]$ for some $c$ and $d$ divisible by $q$. By Lemma \ref{lemma2.4}, $t^{4q^{2}}[a,b]=S_{0,0}t^{2q^{2}}[c,d]$. By Theorem \ref{theorem2.1} and Lemma \ref{lemma2.5}, $t^{2q^{2}}[c,d]$ is $[c+q,d]$ or an earlier monomial. So $t^{4q^{2}}[a,b]$ equals or precedes $[2c+2q,2d]=[a+2q, b-q]$, and this precedes $[a+q,b]$. If $\frac{a}{q}$ is odd, then $[a,b]= S_{q,q}[c,d]$ for some $c$ and $d$ divisible by $q$. By Lemma \ref{lemma2.4}, $t^{4q^{2}}[a,b]=S_{q,0}t^{2q^{2}}[c,d]$. By Theorem \ref{theorem2.1} and Lemma \ref{lemma2.5}, $t^{2q^{2}}[c,d]$ is $[c+q,d]$ or an earlier monomial. So $t^{4q^{2}}[a,b]$ equals or precedes $[2c+3q,2d]=[a+2q, b-q]$, and this precedes $[a,b+q]$.\qed
\end{proof}

We now drop the assumption that $q$ divides $a$ and $b$.

\begin{theorem}
\label{theorem2.7} \hspace{1em}\\
\vspace{-3ex}
\begin{enumerate}
\item[(1)] If $\left\lfloor\frac{a}{q}\right\rfloor$ is odd, $t^{2q^{2}}[a,b]$ precedes $[a+q,b]$.
\item[(2)] If $\left\lfloor\frac{b}{q}\right\rfloor$ is odd, $t^{4q^{2}}[a,b]$ precedes $[a,b+q]$.
\end{enumerate}
\end{theorem}

\begin{proof}
To prove (1), write $a$ as $a_{1}+a_{2}$, $b$ as $b_{1}+b_{2}$ where $q$ divides $a_{1}$ and $b_{1}$, and both $a_{2}$ and $b_{2}$ are in $[0,q)$. By Lemma \ref{lemma2.5}, $t^{2q^{2}}[a_{1},b_{1}]=[c_{1},d_{1}]$ where $q$ divides $c_{1}$ and $d_{1}$, and $[c_{1},d_{1}]$ precedes $[a_{1}+q,b_{1}]$. Now $g(a)=g(a_{1})+g(a_{2})$ while $g(b)=g(b_{1})+g(b_{2})$. If we let $c=c_{1}+a_{2}$, $d=d_{1}+b_{2}$, then $g(c)=g(c_{1})+g(a_{2})$, $g(d)=g(d_{1})+g(b_{2})$. It follows immediately that  $t^{2q^{2}}[a,b]=[c,d]$. And since $[c_{1},d_{1}]$ precedes $[a_{1}+q,b_{1}]$, $[c,d]$ precedes $[a+q,b]$. The proof of (2) is the same.\qed
\end{proof}

Combining Theorems \ref{theorem2.1} and \ref{theorem2.7} we get:

\begin{corollary}
\label{corollary2.8}
$t^{2q^{2}}[a,b]$ is $[a+q,b]$ or an earlier monomial, while $t^{4q^{2}}[a,b]$ is $[a,b+q]$ or an earlier monomial.
\end{corollary}

\section{Theorems relating to the recursion $\bm (\star)$}
\label{section3}

We now complete the proof, outlined at the end of section \ref{section1}, of Theorem \ref{theorem3.1}. As we indicated, the argument is a 2 stage one. In the first stage we assume that $(1\gamma)$ holds for $n<8q^{2}$, and argue inductively on $k$ to show that $(1\gamma)$ holds for $k$ in $(8q^{2}, 16q^{2})$. Recall that $n=k-8q^{2}$, that $t^{n}=[a,b]$ and that $t^{k}=[a+2q, b]$. As we've seen, it suffices to show that $t^{8q^{2}}A_{n}$ and $t^{2q^{2}}A_{n+2q^{2}}$ are each sums of monomials preceding $[a+2q,b]$. Now since $n<8q^{2}$, $(1\gamma)$ holds for $n$, and $A_{n}$ is a sum of $[c,d]$ preceding $[a,b]$. By Corollary \ref{corollary2.8}, with $q$ replaced by $2q$, $t^{8q^{2}}[c,d]$ is $[c+2q,d]$ or an earlier monomial, and so precedes $[a+2q,b]$. Also, $t^{n+2q^{2}}=t^{2q^{2}}[a,b]$ which, by Corollary \ref{corollary2.8}, is $[a+q,b]$ or an earlier monomial.  Since $n+2q^{2}<k$, $(1\gamma)$ holds for $n+2q^{2}$, and $A_{n+2q^{2}}$ is a sum of $[c,d]$ preceding $[a+q,b]$. Corollary \ref{corollary2.8} shows that $t^{2q^{2}}[c,d]$ is $[c+q,d]$ or an earlier monomial.  Then $[c+q,d]$ precedes $[a+2q,b]$, and we're done. Turning to stage 2 we assume that $(1\gamma)$ holds for $n<16q^{2}$, and argue inductively on $k$ to show that $(1\gamma)$ holds for $k$ in $(16q^{2},32q^{2})$. Now $n=k-16q^{2}$, $t^{n}=[a,b]$ and $t^{k}=[a,b+2q]$ It sufices to show that $t^{16q^{2}}A_{n}$ and $t^{4q^{2}}A_{n+4q^{2}}$ are sums of monomials preceding $[a,b+2q]$, and we follow the argument of stage 1, once again using Corollary \ref{corollary2.8} repeatedly.  We have proved:

\begin{theorem}
\label{theorem3.1}
Suppose $A_{n}$, $n$ odd and $>0$, are in $tZ/2[t^{2}]$ and satisfy the recursion $A_{n+8r}=t^{8r}A_{n}+t^{2r}A_{n+2r}$ for some $r$ which is a power of $2$. If $A_{n}$ is a sum of monomials preceding $t^{n}$ whenever $n<8r^{2}$, then for all $n$, $A_{n}$ is such a sum.
\end{theorem}

\begin{corollary}
\label{corollary3.2}
Let $\gamma_{n}$ be defined by $\gamma_{n+8}=t^{8}\gamma_{n}+t^{2}\gamma_{n+2}$, $\gamma_{1}=0$, $\gamma_{3}=0$, $\gamma_{5}=t^{3}$, $\gamma_{7}=t^{5}$. Then, for all $n$, $\gamma_{n}$ is a sum of monomials preceding $t^{n}$.
\end{corollary}

For $t^{3}$, $t^{5}$ and $t^{7}$ are $[1,0]$, $[0,1]$ and $[1,1]$, so the hypotheses of Theorem \ref{theorem3.1} apply.

We now introduce 3 variants, $(1\beta)$, $(1\delta)$ and $(1\alpha)$ of $(1\gamma)$, and derive parallel results for them.

\begin{enumerate}
\item[$(1\beta)$] $A_{n}$ is $t^{n}\ +$ a sum of earlier monomials.
\item[$(1\delta)$] Suppose $t^{n}=[a,b]$. If $a>0$, $A_{n}$ is a sum of monomials preceding $[a-1,b]$. If $a=0$, $A_{n}$ is a sum of monomials equal to or preceding $[0,b-1]$.
\item[$(1\alpha)$] Suppose $t^{n}=[a,b]$. If $a>0$, $A_{n}=[a-1,b]+$  a sum of earlier monomials. If $a=0$, $A_{n}$ is a sum of monomials equal to or preceding $[0,b-1]$.
\end{enumerate}

\begin{theorem}
\label{theorem3.3}
Suppose $A_{n}$ are as in Theorem \ref{theorem3.1}. If $A_{n}$ satisfies $(1\beta)$ whenever $n<8r^{2}$, every $A_{n}$ satisfies $(1\beta)$.
\end{theorem}

\begin{corollary}
\label{corollary3.4}
Let $\beta_{n}$ be defined by $\beta_{n+8}=t^{8}\beta_{n}+t^{2}\beta_{n+2}$, $\beta_{1}=t$, $\beta_{3}=t^{3}$, $\beta_{5}=t^{5}$, $\beta_{7}=t^{7}+t^{3}$. Then, for all $n$, $\beta_{n}$ is $t^{n}\ +$ a sum of earlier monomials.
\end{corollary}

The proof of Theorem \ref{theorem3.3} uses the same 2 step process. Suppose we're in the first stage and $k$ is in $(8r^{2}, 16r^{2})$. Again, $n=k-8q^{2}$, $t^{n}=[a,b]$ and $t^{k}=[a+2q, b]$. Also $A_{k}=t^{8q^{2}}A_{n} + t^{2q^{2}}A_{n+2q^{2}}$. So it suffices to show that $t^{8q^{2}}A_{n}$ is $[a+2q,b]\ +$ a sum of earlier monomials, while $t^{2q^{2}}A_{n+2q^{2}}$ is a sum of monomials preceding $[a+2q,b]$. The first of these is easy---$A_{n}$ is $[a,b]\ +$ a sum of earlier monomials $[c,d]$, $t^{8q^{2}}[a,b]=[a+2q,b]$ and $t^{8q^{2}}[c,d]=[c+2q,d]$, which precedes $[a+2q, b]$. For the second we use Corollary \ref{corollary2.8} to see that $t^{n+2q^{2}}$ is $[a+q,b]$ or an earlier monomial. The induction then shows that $A_{n+2q^{2}}$ is either a sum of $[c,d]$ preceding $[a+q,b]$ or $[a+q,b]+$ such a sum. In the first case each $t^{2q^{2}}A_{n+2q^{2}}$ is $[c+q,d]$ or an earlier monomial and so precedes $[a+2q,b]$. Now in the second case we must have $t^{2q^{2}}[a,b]=[a+q,b]$; Theorem \ref{theorem2.7} tells us that this only occurs when $\left\lfloor\frac{a}{q}\right\rfloor$ is even. But when this happens, $\left\lfloor\frac{a+q}{q}\right\rfloor$ is odd, so by Theorem \ref{theorem2.7}, $t^{2q^{2}}[a+q,b]$ precedes $[a+2q,b]$. So in all cases $t^{2q^{2}}A_{n+2q^{2}}$ is a sum of the desired sort. The argument in the second stage is almost identical.  Now $n=k-16q^{2}$, $t^{n}=[a,b]$ and $t^{k}=[a,b+2q]$.  Also $A_{k}=t^{16q^{2}}A_{n}+t^{4q^{2}}A_{n+4q^{2}}$. So it suffices to show that $t^{16q^{2}}A_{n}$ is $[a,b+2q]\ +$ a sum of earlier monomials, and that $t^{4q^{2}}A_{n+4q^{2}}$ is a sum of monomials preceding $[a,b+2q]$. The argument to establish these facts is the same as that used in stage 1, again employing Corollary \ref{corollary2.8} and Theorem \ref{theorem2.7}.

\begin{theorem}
\label{theorem3.5}
Suppose $A_{n}$ are as in Theorem \ref{theorem3.1}. If $A_{n}$ satisfies $(1\delta)$ whenever $n<8r^{2}$, every $A_{n}$ satisfies $(1\delta)$.
\end{theorem}

\begin{corollary}
\label{corollary3.6}
Let $\delta_{n}$ be defined by $\delta_{n+8}=t^{8}\delta_{n}+t^{2}\delta_{n+2}$, $\delta_{1}=\delta_{3}=\delta_{5}=0$, $\delta_{7}=t^{3}$. Then if $t^{n}=[a,b]$, $a>0$, $\delta_{n}$ is a sum of monomials preceding $[a-1,b]$.  And if $t^{n}=[0,b]$, $\delta_{n}$ is a sum of monomials equal to or preceding $[0,b-1]$.
\end{corollary}

We use the 2 stage process to prove Theorem \ref{theorem3.5}. Suppose we're in the first stage. Again, $k$ is in $(8q^{2}, 16q^{2})$, $n=k-8q^{2}$, $t^{n}=[a,b]$, and $t^{k}=[a+2q, b]$.  When $a>0$ we argue exactly as in the proof of Theorem \ref{theorem3.1}. Suppose $a=0$, so that $t^{n}=[0,b]$, and  $t^{k}=[2q, b]$. It's enough to show that $t^{8q^{2}}A_{n}$ and $t^{2q^{2}}A_{n+2q^{2}}$ are each sums of monomials preceding $[2q-1,b]$. Since $n<8q^{2}$, $A_{n}$ satisfies $(1\delta)$, and is a sum of $[c,d]$ equal to or preceding $[0,b-1]$. Furthermore, $t^{8q^{2}}[c,d]$ is $[c+2q,d]$ or an earlier monomial, and so is equal to or precedes $[2q,b-1]$ which in turn precedes $[2q-1,b]$. Also, $t^{n+2q^{2}}=[q,b]$. Since $n+2q^{2}<k$, $A_{n+2q^{2}}$ satisfies $(1\delta)$, and is a sum of $[c,d]$ preceding $[q-1,b]$. Then $t^{2q^{2}}[c,d]$ is equal to or precedes $[c+q,d]$ and so precedes $[2q-1,b]$. We turn to stage 2. Now $k$ is in $(16q^{2},32q^{2})$, $n=k-16q^{2}$, $t^{n}=[a,b]$, $t^{k}=[a,b+2q]$. When $a>0$, the argument is just as in Theorem \ref{theorem3.1}. Suppose $a=0$, so that $t^{n}=[0,b]$, $t^{n+4q^{2}}$ is $[0,b+q]$ or an earlier monomial, $t^{k}=[0,b+2q]$.  It will suffice to show that $t^{16q^{2}}A_{n}$ and $t^{4q^{2}}A_{n+4q^{2}}$ are each sums of monomials equal to or preceding $[0,b+2q-1]$. Now since $n<16q^{2}$, $A_{n}$ satisfies $(1\delta)$ and is a sum of $[c,d]$ equal to or preceding $[0,b-1]$. Then $t^{16q^{2}}[c,d]$ is $[c,d+2q]$ or an earlier monomial and so is equal to or precedes $[0,b+2q-1]$. Also $n+4q^{2}<k$, and our induction shows that $A_{n+4q^{2}}$ is a sum of $[c,d]$ preceding $[0,b+q-1]$ (both when $t^{4q^{2}}$ is $[0,b+q]$ and when it is earlier). Then  $t^{4q^{2}}[c,d]$ is $[c,d+q]$ or an earlier monomial, and so precedes $[0,b+2q-1]$, completing the proof.

\begin{theorem}
\label{theorem3.7}
Suppose $A_{n}$ are as in Theorem \ref{theorem3.1}. If $A_{n}$ satisfies $(1\alpha)$ whenever $n<8r^{2}$, every $A_{n}$ satisfies $(1\alpha)$.
\end{theorem}

\begin{corollary}
\label{corollary3.8}
(This is the purely formal algebraic part of Proposition 4.3 of \cite{6}.)  Let $\alpha_{n}$ be defined by $\alpha_{n+8}=t^{8}\alpha_{n}+t^{2}\alpha_{n+2}$, $\alpha_{1}=0$, $\alpha_{3}=t$, $\alpha_{5}=0$, $\alpha_{7}=t^{5}$. Then if $t^{n}=[a,b]$ with $a>0$, $\alpha_{n}=[a-1,b]\ +$ a sum of earlier monomials.  And if $t^{n}=[0,b]$, $\alpha_{n}$ is a sum of monomials equal to or preceding $[0,b-1]$.
\end{corollary}

Again we employ the 2 stage process. In the first stage, $k$ is in $(8q^{2}, 16q^{2})$, $n=k-8q^{2}$, $t^{n}=[a,b]$ and  $t^{k}=[a+2q, b]$.  When $a>0$ we proceed as in the proof of Theorem \ref{theorem3.3}, showing that $t^{8q^{2}}A_{n}$ is $[a+2q-1,b]\ +$ a sum of earlier monomials, while $t^{2q^{2}}A_{n+2q^{2}}$ is a sum of monomials preceding $[a+2q-1,b]$. When $a=0$, the situation is different. Since $t^{n}=[0,b]$, $t^{n+2q^{2}}=[q,b]$; see Theorem \ref{theorem2.1}. Now $n+2q^{2}<k$, so by induction $A_{n+2q^{2}}$ is $[q-1,b]\ +$ a sum of earlier monomials $[c,d]$. By Theorem \ref{theorem2.1}, $t^{2q^{2}}[q-1,b]=[2q-1,b]$, while $t^{2q^{2}}[c,d]$, by Corollary \ref{corollary2.8}, is $[c+q,d]$ or an earlier monomial, and so precedes $[2q-1,b]$. If remains to show that $t^{8q^{2}}A_{n}$ is a sum of monomials preceding $[2q-1,b]$. Since $n<8q^{2}$, $A_{n}$ is a sum of $[c,d]$ equal to or preceding $[0,b-1]$, and $t^{8q^{2}}[c,d]$ is $[c+2q,d]$ or an earlier monomial and so is equal to or precedes $[2q,b-1]$ which precedes $[2q-1,b]$. In the second stage, $k$ is in $(16q^{2}, 32q^{2})$, $n=k-16q^{2}$, $t^{n}=[a,b]$ and $t^{k}=[a,b+2q]$.  Now $A_{k}=t^{16q^{2}}A_{n}+t^{4q^{2}}A_{n+4q^{2}}$. When $a>0$ we proceed as in the proof of Theorem \ref{theorem3.3}, showing that $t^{16q^{2}}A_{n}$ is $[a-1,b+2q]\ +$ a sum of earlier monomials, while $t^{4q^{2}}A_{n+4q^{2}}$ is a sum of monomials preceding $[a-1,b+2q]$. Suppose finally that $a=0$. Now $t^{k}$ is $[0,b+2q]$ and we need to show that $A_{k}$ is a sum of monomials equal to or preceding $[0,b+2q-1]$.  It's enough to prove this separately for $t^{16q^{2}}A_{n}$ and $t^{4q^{2}}A_{n+4q^{2}}$. Since $t^{n}=[0,b]$ and $n<16q^{2}$, $A_{n}$ is a sum of $[c,d]$, each of which is $[0,b-1]$ or earlier. Then $t^{16q^{2}}[c,d]$ is $[c,d+2q]$ or an earlier monomial, and so is equal to or precedes $[0,b+2q-1]$. And $t^{n+4q^{2}}=t^{4q^{2}}[0,b]$ is $[0,b+q]$ or an earlier monomial. Using the fact that $n+4q^{2}<k$, we conclude that $A_{n+4q^{2}}$ is a sum of monomials equal to or preceding $[0,b+q-1]$ (both when $t^{n+4q^{2}}$ is $[0,b+q]$ and when it is earlier). Since $t^{4q^{2}}[c,d]$ is $[c,d+q]$ or an earlier monomial, and so is equal to or earlier than $[0,b+2q-1]$, we're done.

\section{A theorem related to the recursion $\bm (\star\star)$}
\label{section4}

In the introduction we described a recursion $(\star\star)$ very much like $(\star)$: $A_{n+16r} = t^{16r}A_{n}+t^{4r}A_{n+4r}+t^{2r}A_{n+2r}$, where $r$ is once again a power of $2$.

\begin{theorem}
\label{theorem4.1}
Suppose $A_{n}$, $n$ odd and $>0$, are in $tZ/2[t^{2}]$, and satisfy the above recursion. If the $A_{n}$ satisfy condition $(1\alpha)$ of the last section whenever $n<16r^{2}$, then they satisfy it for all $n$.
\end{theorem}

\begin{corollary}
\label{corollary4.2}
Suppose $A_{n}$ satisfy $A_{n+16}=t^{16}A_{n}+t^{4}A_{n+4}+t^{2}A_{n+2}$, and that $A_{1}$, $A_{3}$, $A_{5}$, $A_{7}$, $A_{9}$, $A_{11}$, $A_{13}$, $A_{15}$ are $0$, $t$, $0$, $t^{5}$, $t^{3}$, $t^{9}+t$, $t^{7}$ and $t^{13}+t^{5}$. Then if $t^{n}=[a,b]$ with $a>0$, $A_{n}$ is $[a-1,b]\ +$ a sum of earlier monomials. And if $t^{n}=[0,b]$, $A_{n}$ is a sum of monomials equal to or preceding $[0,b-1]$.
\end{corollary}

(Corollary \ref{corollary4.2} appears as Theorem 2.22 of \cite{2}, and plays an important part in establishing the results of that paper.)

We now give the proof of Theorem \ref{theorem4.1}. Note first that if the $A_{n}$ satisfy $(\star\star)$ for some $r$ then they satisfy it for every larger power of $2$. Also, each of $16r, 32r, 64r, \ldots $ is either $32q^{2}$ or $16q^{2}$ for some $q$ which is a power of $2$. So we can use the 2 stage argument of section \ref{section3}.

In the first stage we know that $(1\alpha)$ holds for $n<32q^{2}$, and want to treat $k$ in $(32q^{2},64q^{2})$. We set $n=k-32q^{2}$, and let $t^{n}=[a,b]$. The argument made earlier with $q$ replaced by $2q$ shows that $t^{k}=[a+4q,b]$. Now $A_{k}=t^{32q^{2}}A_{n}+t^{8q^{2}}A_{n+8q^{2}}+t^{4q^{2}}A_{n+4q^{2}}$. The proof of Theorem \ref{theorem3.7} shows that the sum of the first 2 of these 3 terms is $[a+4q-1,b]\ +$ a sum of earlier monomials. It remains to show that $t^{4q^{2}}A_{n+4q^{2}}$ is a sum of monomials preceding $[a+4q-1,b]$. Now $t^{n+4q^{2}}=t^{4q^{2}}[a,b]$ is $[a,b+q]$ or an earlier monomial. Also, $n+4q^{2}<k$. So by induction, when $a>0$, $A_{n+4q^{2}}$ is a sum of $[c,d]$ equal to or preceding $[a-1,b+q]$, and when $a=0$, it is a sum of $[c,d]$ equal to or preceding $[0,b+q-1]$. Then $t^{4q^{2}}[c,d]$ is $[c,d+q]$ or an earlier monomial, and so precedes $[a+4q-1,b]$.

In the second stage, we know that $(1\alpha)$ holds for $n<16q^{2}$, and want to treat $k$ in $(16q^{2},32q^{2})$. We set $n=k-16q^{2}$, and let $t^{n}=[a,b]$.  Then $t^{k}=[a,b+2q]$. Now $A_{k}=t^{16q^{2}}A_{n}+t^{4q^{2}}A_{n+4q^{2}}+t^{2q^{2}}A_{n+2q^{2}}$. The proof of Theorem \ref{theorem3.7} shows that $t^{16q^{2}}A_{n}+t^{4q^{2}}A_{n+4q^{2}}$ is a sum of the desired sort, and it remains to show that when $a>0$, $t^{2q^{2}}A_{n+2q^{2}}$ is a sum of monomials preceding $[a-1,b+2q]$, while when $a=0$, $t^{2q^{2}}A_{n+2q^{2}}$ is a sum of monomials preceding or equal to $[0, b+2q-1]$. Now $t^{n+2q^{2}}=t^{2q^{2}}[a,b]$ is $[a+q,b]$ or an earlier monomial. Since $n+2q^{2}<k$, $A_{n+2q^{2}}$ is a sum of $[c,d]$ equal to or preceding $[a+q-1,b]$. Now $t^{2q^{2}}[c,d]$ is $[c+q,d]$ or an earlier monomial, and so is equal to or precedes $[a+2q-1,b]$. But when $a>0$, $[a+2q-1,b]$ precedes $[a-1,b+2q]$, while when $a=0$, $[2q-1,b]$ precedes $[0,b+2q-1]$. This completes the proof.

\section{The polynomials $\bm P_{n}$. Statement of Lemma 5.5}
\label{section5}

Recall that in Corollaries \ref{corollary3.8}, \ref{corollary3.4}, \ref{corollary3.2} and \ref{corollary3.6} we defined, for $n$ odd and $>0$, elements $\alpha_{n}$, $\beta_{n}$, $\gamma_{n}$ and $\delta_{n}$ of $V=tZ/2[t^{2}]$. Explicitly, these all satisfy $A_{n+8}=t^{8}A_{n}+t^{2}A_{n+2}$, and the initial conditions are:

\renewcommand{\arraystretch}{1.0} 
\begin{center}
\begin{tabular}{lcl}
$\alpha_{1}, \alpha_{3}, \alpha_{5}, \alpha_{7}$ &are& $0,t,0,t^{5}$\\
$\beta_{1}, \beta_{3}, \beta_{5}, \beta_{7}$ &are& $t,t^{3},t^{5},t^{7}+t^{3}$\\
$\gamma_{1}, \gamma_{3}, \gamma_{5}, \gamma_{7}$ &are& $0,0,t^{3},t^{5}$\\
$\delta_{1}, \delta_{3}, \delta_{5}, \delta_{7}$ &are& $0,0,0,t^{3}$
\end{tabular}
\end{center}

Now let $w$ be an indeterminate over $Z/2$. If $k\equiv 1,3,7\mbox{ or }9\bmod{20}$, then $k=10n-9, 10n-7, 10n-3\mbox{ or }10n-1$ for some odd $n$, and we define $P_{k}$ in $Z/2[w]$ as follows:

\begin{center}
\begin{tabular}{lcl}
$P_{10n-9}$ &=& $w^{-3}\alpha_{n}(w^{10})+w^{-7}\delta_{n}(w^{10})$\\
$P_{10n-7}$ &=& $w^{-9}\beta_{n}(w^{10})+w^{-1}\delta_{n}(w^{10})$\\
$P_{10n-3}$ &=& $w^{-1}\alpha_{n}(w^{10})+w^{-9}\gamma_{n}(w^{10})$\\
$P_{10n-1}$ &=& $w^{-7}\beta_{n}(w^{10})+w^{-3}\gamma_{n}(w^{10})$
\end{tabular}
\end{center}

\begin{theorem}
\label{theorem5.1}
$P_{k+80} = w^{80}P_{k}+w^{20}P_{k+20}$. Furthermore, $P_{1}$, $P_{3}$, $P_{7}$, $P_{9}$, $P_{21}$, $P_{23}$, $P_{27}$, $P_{29}$, $P_{41}$, $P_{43}$, $P_{47}$, $P_{49}$, $P_{61}$, $P_{63}$, $P_{67}$, $P_{69}$ are:

$0$, $w$, $0$, $w^{3}$, $w^{7}$, $w^{21}$, $w^{9}$, $w^{23}$, $0$, $w^{41}$, $w^{21}$, $w^{43}+w^{27}$, $
w^{47}+w^{23}$, $w^{61}+w^{29}+w^{21}$, \linebreak $w^{49}+w^{41}$ and $w^{63}+w^{47}+w^{23}$.
\end{theorem}

\begin{remark}
\label{remark5.2}
If $V^{\prime}\subset Z/2[w]$ is spanned by the $w^{k}$ with $k \equiv 1,3,7 \mbox{ or } 9\bmod{20}$, then each $P_{k}$ lies in $V^{\prime}$.
\end{remark}

The recurrence satisfied by the $P_{k}$ is an immediate consequence of the recurrence satisfied by the $\alpha_{n}$, $\beta_{n}$, $\gamma_{n}$ and $\delta_{n}$. And the initial values are immediately verified.

I will prove a result analogous to that of Corollary \ref{corollary3.8} for the $P_{k}$. This involves putting a linear order on the monomials that lie in $V^{\prime}$.

\begin{definition}
\label{def5.3}
$\varphi : V\rightarrow V^{\prime}$ is the following $Z/2$-linear map:
\begin{enumerate}
\item[(a)] $\varphi$ takes $t,t^{3},t^{5},t^{9}$ to $w,w^{3},w^{7},w^{9}$.
\item[(b)] $\varphi$ takes $t^{7},t^{11},t^{13},t^{15}$ to $w^{21},w^{27},w^{23},w^{29}$.
\item[(c)] $\varphi(t^{16}f)=w^{40}\varphi(f)$.
\end{enumerate}
\end{definition}

Evidently $\varphi$ maps $V$ bijectively to $V^{\prime}$.

\begin{definition}
\label{def5.4}
If $(a,b)$ lies in $N\times N$, then $\langle a,b\rangle$ in $V^{\prime}$ is the monomial $\varphi[a,b]$, where $[a,b]$ is $t^{1+2g(a)+4g(b)}$.
\end{definition}

The $\langle a,b\rangle$  run over all the monomials in $w$ lying in $V^{\prime}$. Just as we had a listing, $[0,0]$, $[1,0]$, $[0,1]$, $[2,0]$, $[1,1]$, $[0,2], \ldots$ of the monomials in $V$, we have a corresponding listing $\langle 0,0\rangle $, $\langle 1,0\rangle $, $\langle 0,1\rangle $, $\langle 2,0\rangle $, $\langle 1,1\rangle $, $\langle 0,2\rangle , \ldots$ of the monomials in $V^{\prime}$.  We again use the language of ``earlier monomial'' or ``preceding monomial''.

The rest of this note is devoted to the following result, which is used to prove the essential Theorem 3.5 of \cite{4}:
\pagebreak

\begin{lemma}
\label{lemma5.5}
Let $P_{k}$ in $V^{\prime}$, $k \equiv 1,3,7 \mbox{ or } 9\bmod{20}$ be as in Theorem \ref{theorem5.1}. Suppose that $w^{k}=\langle a,b\rangle$. Then
\begin{enumerate}
\item[(1)] If $a>0$, $P_{k}=\langle a-1,b\rangle\ +$ a sum of earlier monomials.
\item[(2)] If $a=0$, $P_{k}$ is a sum of monomials preceding or equal to $\langle 0,b-1\rangle$.
\end{enumerate}
\end{lemma}

\section{The operators $\bm s_{i,j}$}
\label{section6}

\begin{lemma}
\label{lemma6.1}
If $n\equiv 1\mod{4}$, the $\varphi$ of Definition \ref{def5.3} takes $t^{4n-3}$, $t^{4n-1}$, $t^{4n+1}$ and $t^{4n+5}$ to $w^{10n-9}$, $w^{10n-7}$, $w^{10n-3}$, $w^{10n-1}$.
\end{lemma}

\begin{proof}
When $n=1$ this is Definition \ref{def5.3}~(a). Replacing $n$ by $n+4$ multiplies both $\varphi(t^{4n+c})$ and $w^{10n+d}$ by $w^{40}$, giving the result.\qed
\end{proof}

\begin{lemma}
\label{lemma6.2}
If $n\equiv 3\mod{4}$, the $\varphi$ of Definition \ref{def5.3} takes $t^{4n-5}$, $t^{4n-1}$, $t^{4n+1}$ and $t^{4n+3}$ to $w^{10n-9}$, $w^{10n-3}$, $w^{10n-7}$, $w^{10n-1}$.
\end{lemma}

\begin{proof}
For $n=3$ this is Definition \ref{def5.3}~(b), and we continue as above.\qed
\end{proof}

\begin{definition}
\label{def6.3}
$V_{1}$ is the subspace of $V$ spanned by $t^{n}$, $n\equiv 1\mod{4}$, $V_{3}$ the subspace spanned by the $t^{n}$, $n\equiv 3\mod{4}$.
\end{definition}

Note that $V_{1}$ is spanned by the $[c,d]$ with $c$ even, while $V_{3}$ is spanned by the $[c,d]$ with $c$ odd.

\begin{definition}
\label{def6.4} \hspace{2em} 
\begin{enumerate}
\item[(1)] If $i$ and $j$ are $\ge 0$, $s_{i,j}:V\rightarrow V^{\prime}$ is the $Z/2$-linear map taking $[c,d]$ to $\langle 2c+i,2d+j\rangle$.
\item[(2)] $s_{-1,1}:V_{3}\rightarrow V^{\prime}$ is the $Z/2$-linear map taking $[c,d]$ to $\langle 2c-1,2d+1\rangle $. (Since $c>0$ this makes sense.)
\end{enumerate}
\end{definition}

\begin{lemma}
\label{lemma6.5}
If $f$ is in $V_{1}$, then $s_{0,0}$, $s_{1,0}$, $s_{0,1}$ and $s_{2,0}$ take $f$ to $w^{-9}f(w^{10})$, $w^{-7}f(w^{10})$, $w^{-3}f(w^{10})$, $w^{-1}f(w^{10})$.
\end{lemma}

\begin{proof}
We may assume $f=t^{n}$ with $n\equiv 1\mod{4}$. Write $f=[c,d]$. Arguing as in Lemma \ref{lemma2.3} we find that $[2c,2d]=t^{4n-3}$, $[2c+1,2d]=t^{4n-1}$, $[2c,2d+1]=t^{4n+1}$ and $[2c+2,2d]=t^{4n+5}$. (For the last of these, note that $c$ is even.)  Applying $\varphi$ and using Lemma \ref{lemma6.1} we get the result.\qed
\end{proof}

\begin{lemma}
\label{lemma6.6}
If $f$ is in $V_{3}$, then $s_{-1,1}$, $s_{1,0}$, $s_{0,1}$ and $s_{1,1}$ take $f$ to $w^{-9}f(w^{10})$, $w^{-3}f(w^{10})$, $w^{-7}f(w^{10})$ and $w^{-1}f(w^{10})$.
\end{lemma}

\begin{proof}
We may assume $f=t^{n}$ with $n\equiv 3\mod{4}$. Write $f=[c,d]$. Since $c$ is odd, $[c,d]=t^{2}[c-1,d]$. So $[c-1,d]=t^{n-2}$, and $[2c-1,2d+1]$ is, by Lemma \ref{lemma2.3}, just $t^{4n-5}$. Lemma \ref{lemma2.3} also shows that $[2c+1,2d]=t^{4n-1}$, $[2c,2d+1]=t^{4n+1}$, $[2c+1,2d+1]=t^{4n+3}$. Now apply $\varphi$ and use Lemma \ref{lemma6.2}.\qed
\end{proof}

\begin{lemma}
\label{lemma6.7}
If $n\equiv 1\mod{4}$, $\beta_{n}$ and $\delta_{n}$ are in $V_{1}$, while $\alpha_{n}$ and $\gamma_{n}$ are in $V_{3}$.  If $n\equiv 3\mod{4}$, $\alpha_{n}$ and $\gamma_{n}$ are in $V_{1}$, while $\beta_{n}$ and $\delta_{n}$ are in $V_{3}$.
\end{lemma}

\begin{proof}
This is evidently true when $n$ is $1$, $3$, $5$ or $7$. In general it follows from the recursion.\qed
\end{proof}

\section{Proof of Lemma 5.5}
\label{section7}

\begin{lemma}
\label{lemma7.1}
Lemma \ref{lemma5.5} holds when $k\equiv 21, 23, 27$ or $29\mod{40}$.
\end{lemma}

\begin{proof}
We may write $k$ as $10n-9$, $10n-7$, $10n-3$ or $10n-1$ with $n\equiv 3\mod{4}$. Let $t^{n}=[c,d]$; note that $c$ is odd.

Suppose first $k\equiv 21\mod{40}$. Then $s_{-1,1}(t^{n})=\langle 2c-1,2d+1\rangle$, which by Lemma \ref{lemma6.6} is $w^{-9}(w^{10n})=w^{k}$. So we need to show that $P_{k}=\langle 2c-2,2d+1\rangle \ +$ a sum of earlier monomials. Now $P_{k}=w^{-3}\alpha_{n}(w^{10})+w^{-7}\delta_{n}(w^{10})$, and Lemma \ref{lemma6.7} tells us that $\alpha_{n}$ and $\delta_{n}$ are in $V_{1}$ and $V_{3}$. By Lemmas \ref{lemma6.5} and \ref{lemma6.6}, $s_{0,1}(\alpha_{n})=w^{-3}\alpha_{n}(w^{10})$, $s_{0,1}(\delta_{n})=w^{-7}\delta_{n}(w^{10})$. So $P_{k}=s_{0,1}(\alpha_{n})+s_{0,1}(\delta_{n})$. Corollaries \ref{corollary3.8} and \ref{corollary3.6} show that $\alpha_{n}=[c-1,d]\ +$ a sum of earlier monomials, while $\delta_{n}$ is a sum of monomials preceding $[c-1,d]$. Since the $s_{i,j}$ preserve the linear order on monomials, $s_{0,1}(\alpha_{n})+s_{0,1}(\delta_{n})=\langle 2c-2,2d+1\rangle\ +$ a sum of earlier monomials, as desired.

Suppose $k\equiv 23\mod{40}$. Then $s_{0,1}(t^{n})=\langle 2c,2d+1\rangle$, which by Lemma \ref{lemma6.6} is $w^{-7}(w^{10n})=w^{k}$. So we need to show that $P_{k}=\langle 2c-1,2d+1\rangle \ +$ a sum of earlier monomials. Now $P_{k}=w^{-9}\beta_{n}(w^{10})+w^{-1}\delta_{n}(w^{10})$, and Lemma \ref{lemma6.7} tells us that $\beta_{n}$ and $\delta_{n}$ are in $V_{3}$. By Lemma \ref{lemma6.6}, $s_{-1,1}(\beta_{n})=w^{-9}\beta_{n}(w^{10})$, while  $s_{1,1}(\delta_{n})=w^{-1}\delta_{n}(w^{10})$. So $P_{k}=s_{-1,1}(\beta_{n})+s_{1,1}(\delta_{n})$. By Corollaries \ref{corollary3.4} and \ref{corollary3.6}, $\beta_{n}=[c,d]\ +$ a sum of earlier monomials, while $\delta_{n}$ is a sum of monomials preceding $[c-1,d]$. So $P_{k}=\langle 2c-1,2d+1\rangle\ +$ a sum of earlier monomials.

Suppose $k\equiv 27\mod{40}$. Then $s_{1,0}(t^{n})=\langle 2c+1,2d\rangle$, which by Lemma \ref{lemma6.6} is $w^{-3}(w^{10n})=w^{k}$. So we need to show that $P_{k}=\langle 2c,2d\rangle \ +$ a sum of earlier monomials. Now $P_{k}=w^{-1}\alpha_{n}(w^{10})+w^{-9}\gamma_{n}(w^{10})$, and Lemma \ref{lemma6.7} tells us that $\alpha_{n}$ and $\gamma_{n}$ are in $V_{1}$. By Lemma \ref{lemma6.5}, $s_{2,0}(\alpha_{n})=w^{-1}\alpha_{n}(w^{10})$, while  $s_{0,0}(\gamma_{n})=w^{-9}\gamma_{n}(w^{10})$. So $P_{k}=s_{2,0}(\alpha_{n})+s_{0,0}(\gamma_{n})$. By Corollaries \ref{corollary3.8} and \ref{corollary3.2}, $\alpha_{n}=[c-1,d]\ +$ a sum of earlier monomials, while $\gamma_{n}$ is a sum of monomials preceding $[c,d]$. So $P_{k}=\langle 2c,2d\rangle\ +$ a sum of earlier monomials.

Finally suppose $k\equiv 29\mod{40}$. Then $s_{1,1}(t^{n})=\langle 2c+1,2d+1\rangle$, which by Lemma \ref{lemma6.6} is $w^{-1}(w^{10n})=w^{k}$. So we need to show that $P_{k}=\langle 2c,2d+1\rangle \ +$ a sum of earlier monomials. Now $P_{k}=w^{-7}\beta_{n}(w^{10})+w^{-3}\gamma_{n}(w^{10})$, and Lemma \ref{lemma6.7} tells us that $\beta_{n}$ and $\gamma_{n}$ are in $V_{3}$ and $V_{1}$. By Lemmas \ref{lemma6.6} and \ref{lemma6.5}, $s_{0,1}(\beta_{n})=w^{-7}\beta_{n}(w^{10})$,   $s_{0,1}(\gamma_{n})=w^{-3}\gamma_{n}(w^{10})$. So $P_{k}=s_{0,1}(\beta_{n})+s_{0,1}(\gamma_{n})$. By Corollaries \ref{corollary3.4} and \ref{corollary3.2}, $\beta_{n}=[c,d]\ +$ a sum of earlier monomials, while $\gamma_{n}$ is a sum of monomials preceding $[c,d]$. It follows that $P_{k}=\langle 2c,2d+1\rangle\ +$ a sum of earlier monomials, completing the proof.\qed
\end{proof}

\begin{lemma}
\label{lemma7.2}
Lemma \ref{lemma5.5} holds when $k\equiv 3$ or $9\mod{40}$.
\end{lemma}

\begin{proof}
We may write $k$ as $10n-7$ or $10n-1$ with $n\equiv 1\mod{4}$. Let $t^{n}=[c,d]$; note that $c$ is even.

Suppose $k\equiv 3\mod{40}$. Then $s_{1,0}(t^{n})=\langle 2c+1,2d\rangle$, which by Lemma \ref{lemma6.5} is $w^{-7}w^{10n}=w^{k}$. So we need to show that $P_{k}=\langle 2c,2d\rangle \ +$ a sum of earlier monomials. Now $P_{k}=w^{-9}\beta_{n}(w^{10})+w^{-1}\delta_{n}(w^{10})$, and Lemma \ref{lemma6.7} tells us that $\beta_{n}$ and $\delta_{n}$ are in $V_{1}$. So by Lemma \ref{lemma6.5}, $s_{0,0}(\beta_{n})=w^{-9}\beta_{n}(w^{10})$ while $s_{2,0}(\delta_{n})=w^{-1}\delta_{n}(w^{10})$. So $P_{k}=s_{0,0}(\beta_{n})+s_{2,0}(\delta_{n})$. By Corollaries \ref{corollary3.4} and \ref{corollary3.6}, $\beta_{n}=[c,d]\ +$ a sum of earlier monomials, while $\delta_{n}$ is a sum of monomials preceding $[c-1,d]$. So $P_{k}=\langle 2c,2d\rangle \ +$ a sum of earlier monomials. 

Suppose $k\equiv 9\mod{40}$. Then $s_{2,0}(t^{n})=\langle 2c+2,2d\rangle$, which by Lemma \ref{lemma6.5} is $w^{-1}w^{10n}=w^{k}$. So we need to show that $P_{k}=\langle 2c+1,2d\rangle \ +$ a sum of earlier monomials. Now $P_{k}=w^{-7}\beta_{n}(w^{10})+w^{-3}\gamma_{n}(w^{10})$, and Lemma \ref{lemma6.7} tells us that $\beta_{n}$ and $\gamma_{n}$ are in $V_{1}$ and $V_{3}$. By Lemmas \ref{lemma6.5} and \ref{lemma6.6}, $s_{1,0}(\beta_{n})=w^{-7}\beta_{n}(w^{10})$ while $s_{1,0}(\gamma_{n})=w^{-3}\gamma_{n}(w^{10})$. So $P_{k}=s_{1,0}(\beta_{n})+s_{1,0}(\gamma_{n})$. By Corollaries \ref{corollary3.4} and \ref{corollary3.2}, $\beta_{n}=[c,d]\ +$ a sum of earlier monomials, while $\gamma_{n}$ is a sum of monomials preceding $[c,d]$. So $P_{k}=\langle 2c+1,2d\rangle\ +$ a sum of earlier monomials, completing the proof.\qed
\end{proof}

\begin{lemma}
\label{lemma7.3}
Lemma \ref{lemma5.5} holds when $k\equiv 1$ or $7\mod{40}$ and $a>0$.
\end{lemma}

\begin{proof}
Now $k=10n-9$ or $10n-3$ with $n\equiv 1\mod{4}$, and $t^{n}=[c,d]$ with $c$ even.

Suppose $k\equiv 1\mod{40}$. Then $s_{0,0}(t^{n})=\langle 2c,2d\rangle$, which by Lemma \ref{lemma6.5} is $w^{-9}w^{10n}=w^{k}$. So $c\ne 0$ and we need to show that $P_{k}=\langle 2c-1,2d\rangle \ +$ a sum of earlier monomials. Now $P_{k}=w^{-3}\alpha_{n}(w^{10})+w^{-7}\delta_{n}(w^{10})$ and by Lemma \ref{lemma6.7}, $\alpha_{n}$ and $\delta_{n}$ are in $V_{3}$ and $V_{1}$. Lemmas \ref{lemma6.6} and \ref{lemma6.5} then show that $P_{k}=s_{1,0}(\alpha_{n})+s_{1,0}(\delta_{n})$. By Corollaries \ref{corollary3.8} and \ref{corollary3.6}, $\alpha_{n}=[c-1,d]\ +$ a sum of earlier monomials, while $\delta_{n}$ is a sum of monomials preceding $[c-1,d]$. So $P_{k}=\langle 2c-1,2d\rangle \ +$ a sum of earlier monomials. 

Suppose $k\equiv 7\mod{40}$. Then $s_{0,1}(t^{n})=\langle 2c,2d+1\rangle$, which is $w^{-3}w^{10n}=w^{k}$. So again $c\ne 0$ and we need to show that $P_{k}=\langle 2c-1,2d+1\rangle \ +$ a sum of earlier monomials. Now $P_{k}=w^{-1}\alpha_{n}(w^{10})+w^{-9}\gamma_{n}(w^{10})$, and both $\alpha_{n}$ and $\gamma_{n}$ are in $V_{3}$. Lemma \ref{lemma6.6} then shows that $P_{k}=s_{1,1}(\alpha_{n})+s_{-1,1}(\gamma_{n})$. By Corollaries \ref{corollary3.8} and \ref{corollary3.2}, $\alpha_{n}=[c-1,d]\ +$ a sum of earlier monomials, while $\gamma_{n}=$ a sum of monomials preceding $[c,d]$. So $P_{k}=\langle 2c-1,2d+1\rangle\ +$ a sum of earlier monomials.\qed
\end{proof}

To complete the proof of Lemma \ref{lemma5.5} we need to consider two final cases---when $w^{k}=\langle 0,b\rangle$ and $k$ is $1$ or $7\bmod 40$.

Suppose $k\equiv 1\mod{40}$. Then $k=10n-9$ with $n\equiv 1\mod{4}$, and as we saw in the last lemma, $t^{n}=[0,d]$ with $b=2d$. Furthermore, $P_{k}=s_{1,0}(\alpha_{n})+s_{1,0}(\delta_{n})$. By Corollaries \ref{corollary3.8} and \ref{corollary3.6}, $\alpha_{n}$ and $\delta_{n}$ are sums of monomials equal to or preceding $[0,d-1]$. Then $P_{k}$ is a sum of monomials equal to or preceding $\langle 1,2d-2\rangle$.  But $\langle 1,2d-2\rangle$ precedes $\langle 0,2d-1\rangle = \langle 0,b-1\rangle$.  Next suppose $k\equiv 7\mod{40}$. Then $k=10n-3$ with $n\equiv 1\mod{4}$, and as we saw in the last lemma, $t^{n}=[0,d]$ with $b=2d+1$. Also, $P_{k}=s_{1,1}(\alpha_{n})+s_{-1,1}(\gamma_{n})$. By Corollaries \ref{corollary3.8} and \ref{corollary3.2}, $\alpha_{n}$ is a sum of monomials equal to or preceding $[0,d-1]$ while $\gamma_{n}$ is a sum of monomials in $V_{3}$ preceding $[0,d]$, and consequently equal to or preceding $[1,d-1]$. Thus $P_{k}$ is a sum of monomials equal to or preceding $\langle 1,2d-1\rangle$.  But $\langle 1,2d-1\rangle$ precedes $\langle 0,2d\rangle = \langle 0,b-1\rangle$, so we're done.



\end{document}